\documentclass{amsart}
\pdfoutput=1
\usepackage[all,pdf]{xy}
\usepackage{xcolor}
\usepackage{amsmath}
\usepackage{amsfonts}
\usepackage{amssymb}
\usepackage{mathrsfs}
\usepackage{amsthm}
\usepackage{pdfpages}
\usepackage{bbm}
\usepackage{hyperref}
\usepackage{bookmark}
\usepackage{fancyhdr}
\hypersetup{
            colorlinks=true,
            linkcolor=blue,
            anchorcolor=blue,
            citecolor=blue}
\numberwithin{equation}{section}

\begin{document}
\newtheorem{theorem}{Theorem}[section]
\newtheorem{corollary}[theorem]{Conjecture}
\newtheorem{lemma}[theorem]{Lemma}
\newtheorem{remark}[theorem]{Remark}
\newtheorem{proposition}[theorem]{Proposition}

\pagestyle{fancy}                   
\fancyhead[RO]{\thepage}
\fancyhead[CO]{ }
\fancyhead[LO]{The fourth moment of holomorphic Hecke cusp forms in shorter intervals}

\fancyhead[RE]{Jinghai Liu}
\fancyhead[CE]{ }
\fancyhead[LE]{\thepage}
\cfoot{ }                                
\rfoot{ }%
\lfoot{ }
\renewcommand{\headrulewidth}{0.1mm} 
\renewcommand{\footrulewidth}{0mm} 
\renewcommand{\headwidth}{\textwidth}

\renewcommand{\thefootnote}{}

\title{The fourth moment of holomorphic Hecke cusp forms in shorter intervals}
\author{Jinghai Liu}
\date{December 2024}

\begin{abstract}
Let $0<c\le 1/4$ be fixed. For $H = K^{\frac{3}{4}+ c}$, we find the average value of the fourth moment of holomorphic Hecke cusp forms of weight varies within $[K,K+H]$, improving a previous result of Khan.
\end{abstract}
\maketitle

\section{Introduction}
\makeatletter\def\Hy@Warning#1{}\makeatother
\footnotetext{The author was partially supported by the National Key R\&D Program of China (No.
2021YFA1000700).}

A central problem in the area of quantum chaos is to understand the behavior of eigenfunctions.
An important example is the case of Maass cusp forms of large Laplace eigenvalue for the space SL${}_2(\mathbb{Z})\backslash \mathbb{H}$.
Let $\phi$ denote such a Maass form, let $\lambda$ denote its Laplace eigenvalue, and let $\phi$ be normalized so that $\int_{{\rm SL}_2(\mathbb{Z}) \backslash \mathbb{H}} |\phi(z)|^2\frac{dxdy}{y^2}=1$.
The measure $\mu_\phi$ is defined by $\mu_\phi = |\phi(z)|^2\frac{dxdy}{y^2}$.
Zelditch \cite{21} has shown that as $\lambda \to \infty$, for a \emph{typical} Maass form $\phi$ the measure $\mu_\phi$ approaches the uniform distribution measure $\frac{3}{\pi}\frac{dxdy}{y^2}$.
This is referred to as ``Quantum Ergodicity''.
Rudnick and Sarnak \cite{31} conjectured that for \emph{every} Maass form $\phi$, the measure $\mu_\phi$ approaches the uniform distribution measure, which is referred to as ``Quantum Unique Ergodicity''.
Lindenstrauss \cite{100} established the conjecture except for the possibility of mass escaping to the cusp which was eliminated later by Soundararajan \cite{30}.
Holowinsky and Soundararajan \cite{3} gave a proof of a holomorphic analog of the quantum unique ergodicity conjecture.
A more general question is the distribution of Hecke--Maass forms for SL${}_2(\mathbb{Z})$.
The random wave conjecture (RWC) introduced by Berry \cite{101}, predicts that for large Laplace eigenvalue, the distribution is close to be random.
One way to formulate this is to study the moments of Hecke--Maass cusp forms.
The fourth moment is of particular interest because of its relation to $L$-functions via Watson's formula (\cite[Theorem 3]{8}).
By assuming the generalized Lindel${\rm \ddot{o}}$f hypothesis, Buttcane and Khan \cite{29} computed the asymptotic for the fourth moment of Hecke--Maass cusp forms of large Laplace eigenvalue for SL${}_2(\mathbb{Z})$, and confirmed a result predicted by RWC.
\par In this paper, we consider the analogous case for holomorphic cusp forms.
Let $f$ be a holomorphic Hecke cusp form of weight $k$ for SL${}_2(\mathbb{Z})$.
For $z = x+iy$ in the upper half plane $\mathbb{H}$, define the rescaled function $F(z) = y^{k/2}f(z)$.
For $1\le p <\infty$, the $L^p$-norm of $F$ is defined to be
\begin{equation}
||F||_p = \left( \int_{{\rm SL}_2(\mathbb{Z}) \backslash \mathbb{H}} |F(z)|^p \frac{dxdy}{y^2} \right)^{1/p}.\label{2eq1.1}
\end{equation}
\par The $L^4$-norm was studied by Blomer, Khan, and Young in \cite{1}. They obtained the first non-trivial upper bound for $||F||_4$, proving with the normalization $||F||_2 = 1$ that as the weight $k\to \infty$, for any $\varepsilon >0$,
\begin{equation}
||F||_4^4\ll_{\varepsilon}k^{1/3+\varepsilon}. \label{2eq1.2}
\end{equation}
However, this seems to be very far from the truth.
The following conjecture was made in the same paper.
\begin{corollary}[\cite{1},Conjecture 1.2]\label{2thm1.1}
With the normalization
\begin{equation}
\frac{3}{\pi}||F||_2^2 = 1, \label{2eq1.4}
\end{equation}
we have
\begin{equation}
\frac{3}{\pi}||F||_4^4 \sim 2 \label{2eq1.5}
\end{equation}
as $k\to \infty$.
\end{corollary}
\par Zenz \cite{25} proved that the fourth moment of holomorphic Hecke cusp forms is bounded by a constant provided that the generalized Riemann hypothesis holds.
Khan established Conjecture \ref{2thm1.1} on average over all $f$ of weight $k \in [K,2K]$ (cf.\cite[Theorem 1.2]{6}).
Our goal is to establish the analogue of Theorem 1.2 in \cite{6} in a small scale setting.
Let $0<c\le 1/4$ be fixed. For $H = K^{\frac{3}{4}+ c}$, with the normalization $||F||_2 = 1$, we have
\begin{theorem}\label{thm1.3}
Let $w$ be a smooth, non-negative function supported on $(0,1)$ with bounded derivatives.
Then for some $\delta$ with $0<\delta\le c$, we have
\begin{equation}
\frac{2}{HW}\sum_{k\equiv 0 \,mod\,2}w\left(\frac{k-K}{H}\right)\frac{12}{k}\sum_{f\in B_k}||F||_4^4 = \frac{6}{\pi} + O(K^{-\delta}),\label{eq1.6}
\end{equation}
where $W = \int_{-\infty}^{\infty}w(x)dx$.
\end{theorem}
\begin{remark}
The limit of $H$ comes from Lemma \ref{2thm3.5} (See Remark \ref{3remark4}, Remark \ref{3remark3}, Remark \ref{3remark2}, Remark \ref{3remark1}).
\end{remark}
\par The main steps of proof are as follows.
By Watson's formula, the proof of Theorem  \ref{thm1.3} is reduced to finding the mean value of a degree $8$ $L$-function at the central point.
We have to find the triple product $L$-function value
\begin{equation}
L\left(\frac{1}{2},f\times f\times g\right) = L\left(\frac{1}{2},{\rm sym}^2 f\times g\right)L\left(\frac{1}{2},g\right)\label{2eq1.8}
\end{equation}
on average over $g\in B_{2k}$, $f\in B_k$, and $k\in [K,K+H]$.
The Petersson trace formula (Lemma \ref{2thm1.3}) leads to the task of bounding a sum like
\begin{equation}
\sum_{n,m,c_1,c_2}S(n,m;c_1)S(n^2,1;c_2)\sum_{K\le k\le K+H,\,2|k}i^kJ_{2k-1}\left(\frac{4\pi \sqrt{nm}}{c_1}\right)J_{k-1}\left(\frac{4\pi n}{c_2}\right),\label{2eq1.9}
\end{equation}
where $n$, $m$, $c_1$, $c_2$ are positive integers whose range of summation depends on $K$.
Then we develop a series of lemmas to deal with the sum of Bessel functions and use Weil's bound to deal with the sum of Kloosterman sums.

\par Throughout, $\varepsilon$ will denote an arbitrarily small positive constants, but not necessarily the same one from one occurrence to the next.
The implied constants in any estimate may depend on $\varepsilon$.
The delta symbol $\delta_P$ will equal $1$ whenever the statement $P$ is true and $0$ whenever it is false.
\section{$L$-functions}
By definition, $f\in B_k$ is an eigenform of every Hecke operator $T_n$.
Let
\begin{equation}
T_nf = a_f(n)n^{\frac{k-1}{2}}f.\label{2eq1.10}
\end{equation}
The rescaled eigenvalues $a_f(n)$ satisfy the Hecke relations
\begin{equation}
a_f(n)a_f(m) = \sum_{d|(m,n)}a_f\left(\frac{nm}{d^2}\right),\label{2eq1.11}
\end{equation}
and Deligne's bound
\begin{equation}
|a_f(n)| \le \tau(n),\label{2eq1.12}
\end{equation}
where $\tau(n)$ is the number of positive divisors of $n$.
\begin{lemma}[Petersson's trace formula]\label{2thm1.3}
\begin{equation}
\frac{2\pi^2}{k-1}\sum_{f\in B_k}\frac{a_f(n)a_f(m)}{L(1,{\rm sym}^2f)} = \delta_{n=m}+2\pi i^{-k}\sum_{c\ge 1}\frac{S(n,m;c)}{c}J_{k-1}\left(\frac{4\pi \sqrt{nm}}{c}\right),\label{2eq1.13}
\end{equation}
where $L(s,{\rm sym}^2f)$ is the symmetric-square L-function, $J_{k-1}$ is the $J$-Bessel function and
\begin{equation}
S(n,m;c) = {\sum_{b\,mod\,c}}^{*}e\left(\frac{nb+m\bar{b}}{c}\right) \label{2eq1.14}
\end{equation}
is the Kloosterman sum.
\end{lemma}
\begin{proof}
See \cite[Proposition 14.5]{5}.
\end{proof}
\par For $(c_1,c_2)=1$, the Kloosterman sum satisfies the multiplicative property
\begin{equation}
S(n,m;c_1c_2) = S(n,m\overline{c_2}^2;c_1)S(n,m\overline{c_1}^2;c_2) \label{2eq1.15}
\end{equation}
and Weil's bound
\begin{equation}
|S(n,m;c)|\le \tau(c)c^{\frac{1}{2}}(n,m,c)^{\frac{1}{2}}.\label{2eq1.16}
\end{equation}

\par We will work with the following entire $L$-functions, defined for $f\in B_k$, $g\in B_{2k}$ and $\mathfrak{R}(s)>1$ by
\begin{align}
L(s,f)&=\sum_{n\ge 1}\frac{a_f(n)}{n^s},\\
L(s,{\rm sym}^2f)&=\sum_{n\ge 1}\frac{A_f(n,1)}{n^s} = \zeta(2s)\sum_{n\ge 1}\frac{a_f(n^2)}{n^s},\label{10eq29}\\
\begin{split}
L(s,{\rm sym}^2f\times g)&=\sum_{n,r\ge 1}\frac{A_f(n,r)a_g(n)}{(nr^2)^s},\\
&= \sum_{n,r,d\ge 1}\mu(d)\frac{A_f(n,1)A_f(r,1)a_g(nd)}{(nr^2d^3)^s}.\label{10eq210}
\end{split}
\end{align}
where $A_f(n,r) = A_f(r,n)$ are Fourier coefficients of the symmetric-square lift of $f$, a cusp form on GL$(3)$.
By the normalization (\ref{2eq1.10}), these $L$-functions have functional equations relating values at $s$ and $1-s$, so that the central point is always $s = 1/2$.
\par We observe as in \cite{1} the decomposition
\begin{equation}
||F||_4^4 = \langle F^2,F^2\rangle = \sum_{g\in B_{2k}}|\langle F^2,G\rangle|^2,\label{2eq1.21}
\end{equation}
which holds because $f^2$ is a cusp form of weight $2k$.
By Watson's formula, we have
\begin{equation}
|\langle F^2,G\rangle|^2 = \frac{\pi^3}{2(2k-1)}\frac{L(s,g)L(s,{\rm sym}^2f\times g)}{L(1,{\rm sym}^2f)^2L(1,{\rm sym}^2g)}.\label{2eq1.22}
\end{equation}
Thus
\begin{equation}
||F||_4^4 = \frac{\pi^3}{2(2k-1)L(1,{\rm sym}^2f)^2}\sum_{g\in B_{2k}}\frac{L(\frac{1}{2},g)L(\frac{1}{2},{\rm sym}^2f\times g)}{L(1,{\rm sym}^2g)}.\label{2eq1.23}
\end{equation}
\par We get a handle on the central values of $L$-functions using approximate functional equations.
\begin{lemma}[Approximate functional equation]\label{2thm1.4}
Let
\begin{equation}
\Gamma_{k,1}(s) = (2\pi)^{-s}\Gamma\left(s+k-\frac{1}{2}\right)\label{2eq1.24}
\end{equation}
and
\begin{equation}
\Gamma_{k,2}(s) = 8(2\pi)^{-3s-3k+\frac{3}{2}}\Gamma\left(s+2k-\frac{3}{2}\right)\Gamma\left(s+k-\frac{1}{2}\right)\Gamma\left(s+\frac{1}{2}\right).\label{2eq1.25}
\end{equation}
For $\xi ,\sigma >0$ and $j = 1,2$, define
\begin{equation}
V_{k,j}(\xi) = \frac{1}{2\pi i}\int_{(\sigma)}\frac{\Gamma_{k,j}(1/2+s)}{\Gamma_{k,j}(1/2)}\xi^{-s}\frac{ds}{s}.\label{2eq1.26}
\end{equation}
For $g\in B_{2k}$ we have
\begin{equation}
L(1/2,g) = 2\sum_{m\ge 1}\frac{a_g(m)}{m^{1/2}}V_{k,1}(m)\label{2eq1.27}
\end{equation}
and for $f\in B_k$ we have
\begin{equation}
L(1/2,{\rm sym}^2f\times g) = 2\sum_{n,r\ge 1}\frac{A_f(n,r)a_g(n)}{(nr^2)^{1/2}}V_{k,2}(nr^2).\label{2eq1.28}
\end{equation}
\end{lemma}
\begin{proof}
See \cite[Lemma 1.4]{6}.
\end{proof}
\par By Stirling's formula, for any $A>0$ and integer $B\ge 0$ we have (cf.\cite[(1.32)]{6})
\begin{equation}
\left(\frac{\xi}{k^j}\right)^BV_{k,j}^{(B)}(\xi) \ll_{A,B}\left(1+\frac{\xi}{k^j}\right)^{-A}.\label{2eq1.32}
\end{equation}
Thus the sums in (\ref{2eq1.27}) and (\ref{2eq1.28}) are essentially supported on $m<k^{1+\varepsilon}$ and $nr^2<k^{2+\varepsilon}$, respectively.
We also note that for $u\in (1,2)$ we have
\begin{equation}
\frac{d^B}{du^B}V_{uK,j}(\xi) \ll_{A,B}\left(1+\frac{\xi}{K^j}\right)^{-A}.\label{2eq1.33}
\end{equation}
by Stirling's formula.
\par We also need
\begin{lemma}\label{2thm1.5}
Given $0<\delta<1/10$, we have for all but $O(k^{100\delta})$ forms in $B_k$ that
\begin{equation}
\begin{split}
L(1,{\rm sym}^2f)^{-1} = &\sum_{d_1,d_2,d_3\ge 1}\frac{\mu(d_1d_2d_3)\mu(d_2)a_f(d_1^2d_2^2)}{d_1d_2^2d_3^3}{\rm exp}\left(-\frac{d_1d_2^2d_3^3}{k^\delta}\right)\\
&+O(k^{-\delta^2+\varepsilon}).
\end{split}
\end{equation}
\end{lemma}
\begin{proof}
See \cite[Lemma 1.5]{6}.
\end{proof}
\section{Bessel functions}
For $l>0$ an integer, the $J$-Bessel function
\begin{equation}
J_l(x) = \int_{-\frac{1}{2}}^{\frac{1}{2}}e(lt)e^{-ix\sin(2\pi t)}dt\label{2eq1.47}
\end{equation}
satisfies the bounds (cf.\cite[Lemma 4.1-4.3]{7})
\begin{equation}
J_l(x) \ll \left\{\begin{array}{ll}
e^{-l} &for\quad x<l/10,\\
{\rm min}(l^{-1/3},|x^2-l^2|^{-1/4}) \quad&for \quad x\ge l/10\end{array}\right.\label{2eq1.48}
\end{equation}
and for all $x\ge 0$ and $v \ge 15$,
\begin{equation}
|J_v(vx)|\le (2\pi v)^{-1/2}(\frac{1}{2}ex)^v.\label{eq1.33}
\end{equation}
\par The value of $J_{k-1}(y)$ on average over integers $k$ divisible by $4$ is well understood.
We need
\begin{lemma}\label{2thm1.7}
Suppose $g$ is a smooth, real function compactly supported on $\mathbb{R}^{+}$ and $a = \pm 1$.
Then
\begin{equation}
4\sum_{l\equiv a(mod\, 4)}g(l)J_l(x) = g(x)+i^{1-a}h(x)+O(xc_3(g))\label{eq1.34}
\end{equation}
where
\begin{equation}
h(x) = \int_0^{\infty}g(\sqrt{2xy})\sin(x+y-\frac{\pi}{4})(\pi y)^{-1/2}dy \label{eq1.35}
\end{equation}
and
\begin{equation}
c_3(g) = \int_{-\infty}^{\infty}|\hat{g}(t)t^3|dt
\end{equation}
where $\hat{g}$ is the Fourier transform
\begin{equation}
\hat{g}(t) = \int_{-\infty}^{\infty}g(y)e(ty)dy
\end{equation}
and the implied constant is absolute.
\end{lemma}
\begin{proof}
See \cite[Lemma 5.8]{4}.
\end{proof}
\begin{remark}
By partial integration, we have
\begin{equation}
h(x)\ll (xK^{-2})^j\label{eq1.36}
\end{equation}
and
\begin{equation}
c_3(g) \ll K^{-3} \label{3eq3.9}
\end{equation}
for any $j\ge 0$ and $supp\, g \subset [K,2K]$ and $g^{(j)} \ll K^{-j}$.
\end{remark}
\par As in \cite{6} we also need to understand the following average value of a product of two Bessel functions:
\begin{equation}
\sum_{k\equiv 0\,mod\,2}i^kh\left(\frac{k-K}{H}\right)J_{k-1}(x)J_{2k-1}(y)\label{2eq1.54}
\end{equation}
for $0<x<K^{2+\varepsilon}$ and $0<y<K^{3/2+\varepsilon}$.
\begin{lemma}\label{2thm1.8}
Let $A>0$ and $h$ be a smooth, non-negative function supported on $(1,2)$ with bounded derivatives.
Let $x$ and $y$ belong to the set
\begin{equation}
\begin{split}
\{0<x\le K^{4/3-\varepsilon},y>0\} &\cup\{K^{4/3-\varepsilon}<x<K^{2-\varepsilon}, 0<y<xK^{-\varepsilon}\}\\
&\cup\{K^{4/3-\varepsilon}<x<K^{2-\varepsilon},y>xK^{\varepsilon}\}.\label{2eq1.55}
\end{split}
\end{equation}
There exist smooth functions $H^{\pm}(\xi_1;\xi_2,\xi_3,\xi_4)$ depending on $A$ with bounded derivatives in any compact set, polynomial in $\xi_2$, $\xi_3$, $\xi_4$, and supported on
\begin{equation}
1\ll |\xi_1|\ll 1,\label{2eq1.56}
\end{equation}
such that
\begin{equation}
\begin{split}
\sum_{k\equiv 0\,mod\,2} & i^kh\left(\frac{k}{K}\right)J_{k-1}(4\pi x)J_{2k-1}(4\pi y)\\
=&\frac{1}{\sqrt{x}}\sum_{\pm}e\left(\pm\left(\frac{y^2}{4x}+2x\right)\right)\left(1\pm\frac{iy}{4x\sqrt{1-(\frac{y}{4x})^2}}\right)H^{\pm}\left(\frac{y\sqrt{1-(\frac{y}{4x})^2}}{2K}\right)\\
&+O_A(K^{-A}),\label{2eq1.57}
\end{split}
\end{equation}
where
\begin{equation}
\begin{split}
H^{\pm} & \left(\frac{y\sqrt{1-(\frac{y}{4x})^2}}{2K}\right)\\
&=H^{\pm}\left(\frac{y\sqrt{1-(\frac{y}{4x})^2}}{2K};\frac{x(1-2(\frac{y}{4x})^2)}{K^2},\frac{y\sqrt{1-(\frac{y}{4x})^2}}{2K^3},\frac{y}{4x}\right)\delta_{\frac{y}{4x}<1}.\label{2eq1.58}
\end{split}
\end{equation}
\par Property (\ref{2eq1.56}) implies that $H^{\pm}\left(\frac{y\sqrt{1-(\frac{y}{4x})^2}}{2K}\right)$ is nonzero only for $x,y\gg K$ and supported on
\begin{equation}
K^2/y^2\ll 1-y/4x\ll K^2/y^2.\label{2eq1.59}
\end{equation}
\end{lemma}
\begin{proof}
See \cite[Lemma 1.8]{6}.
\end{proof}
\par In the remaining ranges, we find suitable upper bounds for (\ref{2eq1.54}).
\begin{lemma}\label{2thm1.9}
For $x$ and $y$ such that $K^{4/3-\varepsilon}<x<K^{2-\varepsilon}$, $xK^{-\varepsilon}<y<xk^{\varepsilon}$ and $h$ be a smooth, non-negative function supported on $(0,1)$ with bounded derivatives, we have
\begin{equation}
\sum_{k\equiv 0\,mod\,2}i^kh\left(\frac{k-K}{H}\right)J_{k-1}(x)J_{2k-1}(y)\ll \frac{H}{\sqrt{xy}}.\label{2eq1.84}
\end{equation}
Furthermore if $x$ and $y$ satisfy
\begin{equation}
|1-\frac{y}{4x}|>\frac{K^{2+\varepsilon}}{x^2},\label{2eq1.85}
\end{equation}
then we have
\begin{equation}
\sum_{k\equiv 0\,mod\,2}i^kh\left(\frac{k-K}{H}\right)J_{k-1}(x)J_{2k-1}(y)\ll_B \frac{x}{K}\left(\frac{x^2}{K^2|4x-y|}\right)^B \label{2eq1.86}
\end{equation}
for any integer $B\ge 0$.
\end{lemma}
\begin{proof}
The first estimate (\ref{2eq1.84}) follows directly from (\ref{2eq1.48}).
For the second estimate, see \cite[Lemma 1.9]{6}.
\end{proof}
\par For very large $x$, we have
\begin{lemma}\label{2thm1.10}
For $x>K^{2-\varepsilon}$, $0<y<K^{2+\varepsilon}$ and $h$ as in Lemma \ref{2thm1.9}, we have
\begin{equation}
\sum_{k\equiv 0\,mod\,2}i^kh\left(\frac{k-K}{H}\right)J_{k-1}(x)J_{2k-1}(y)\ll K^{-1+\varepsilon}+HK^{-11/6}.\label{2eq1.90}
\end{equation}
\end{lemma}
\begin{proof}
For $u\in (1,2)$, define the following function of $u$:
\begin{equation}
\begin{split}
z(u) = &x\left(1-\frac{(uK-1)^2}{x^2}\right)^{1/2}\\
&+(uK-1){\rm tan}^{-1}\left(\left(\frac{uK-1}{x}\right)\left(1-\frac{(uK-1)^2}{x^2}\right)^{-1/2}\right)\\
=&x+O(K^{\varepsilon})\label{2eq1.91}
\end{split}
\end{equation}
and
\begin{equation}
h_2(u) = \frac{\sin z(u)-\cos z(u)}{\sqrt{z(u)-\frac{(k-1)\pi}{2}}}.\label{2eq1.92}
\end{equation}
\par By \cite[Lemma 1.6]{6} we have
\begin{equation}
\begin{split}
J_{k-1}(x)& = \sqrt{\frac{2}{\pi(z(\frac{k}{K})-\frac{(k-1)\pi}{2})}}{\rm cos}\left(z\left(\frac{k}{K}\right)-\frac{(k-1)\pi}{2}-\frac{\pi}{4}\right)+O(k^{-4/3})\\
&=\frac{i^k}{\sqrt{\pi}}h_2\left(\frac{k}{K}\right)+O(K^{-4/3}).\label{2eq1.94}
\end{split}
\end{equation}
Thus the left hand side of (\ref{2eq1.90}) is bounded by
\begin{equation}
\left|\sum_{k\equiv0\,mod\,2}h\left(\frac{k-K}{H}\right)h_2\left(\frac{k}{K}\right) J_{2k-1}(y)\right|+\frac{1}{K^{4/3}}\sum_{k\equiv0\,mod\,2}h\left(\frac{k-K}{H}\right)|J_{k-1}(y)|.\label{2eq1.95}
\end{equation}
By (\ref{2eq1.91})-(\ref{2eq1.92}) and Lemma \ref{2thm1.7}, the first term above is bounded by $K^{-1+\varepsilon}$.
By (\ref{2eq1.48}), the second term is bounded by $HK^{-11/6}$.
\end{proof}
\section{The main term}
By (\ref{2eq1.23}) and the approximate functional equation (Lemma \ref{2thm1.4}) we have
\begin{equation}
\begin{split}
||F||_4^4 =& \frac{2\pi^3}{(2k-1)L(1,{\rm sym}^2f)^2}\sum_{m,n,r\ge 1}\frac{A_f(n,r)}{(mnr^2)^{1/2}}V_{k,1}(m)V_{k,2}(nr^2)\\
&\times \sum_{g\in B_{2k}}\frac{a_g(m)a_g(n)}{L(1,{\rm sym}^2g)},\label{2eq2.1}
\end{split}
\end{equation}
\par Now by Petersson's trace formula (Lemma \ref{2thm1.3}), we get
\begin{equation}
\begin{split}
||F||_4^4 =&\frac{\pi}{L(1,{\rm sym}^2f)^2}\left(\sum_{n,r\ge 1}\frac{A_f(n,r)V_{k,1}(n)V_{k,2}(nr^2)}{nr} \right.\\
&+ 2\pi\sum_{m,n,r\ge 1}\frac{A_f(n,r)V_{k,1}(m)V_{k,2}(nr^2)}{(mnr^2)^{1/2}}
 \left. \sum_{c\ge 1}\frac{S(n,m;c)}{c}J_{2k-1} \left( \frac{4\pi\sqrt{nm}}{c} \right)  \right).\label{2eq2.2}
\end{split}
\end{equation}
We have
\begin{equation}
\begin{split}
\sum_{n,r\ge 1} &\frac{A_f(n,r)V_{k,1}(n)V_{k,2}(nr^2)}{nr}\\
=&\frac{1}{2\pi i}\int_{(2)}\frac{1}{2\pi i}\int_{(2)}\frac{\Gamma_{k,1}(1/2+s_1)}{\Gamma_{k,1}(1/2)}\frac{\Gamma_{k,2}(1/2+s_2)}{\Gamma_{k,2}(1/2)}\\
&\times \frac{L(1+s_1+s_2,{\rm sym}^2f)L(1+2s_2,{\rm sym}^2f)}{\zeta(2+2s_1+2s_2)}\frac{ds_1}{s_1}\frac{ds_2}{s_2}\label{2eq2.3}
\end{split}
\end{equation}
on using (\ref{2eq1.32}) and the identity
\begin{equation}
\sum_{n,r\ge 1}\frac{A_f(n,r)}{n^sr^w} = \frac{L(s,{\rm sym}^2f)L(w,{\rm sym}^2f)}{\zeta(s+w)}\label{2eq2.4}
\end{equation}
of \cite[Proposition 6.6.3]{2}, valid for $\mathfrak{R}(s), \mathfrak{R}(w)>1$.
We shift the lines of integration in (\ref{2eq2.3}) to $\mathfrak{R}(s_1) =\mathfrak{R}(s_2) = -1/4$, pick up residues at $\mathfrak{R}(s_1) = \mathfrak{R}(s_2) = 0$ and bound the new integral using (\ref{2eq1.32}). The result is that
\begin{equation}
\sum_{n,r\ge 1} \frac{A_f(n,r)V_{k,1}(n)V_{k,2}(nr^2)}{nr} = \frac{6}{\pi^2}L(1,{\rm sym}^2f)^2+O(k^{-1/2}).\label{2eq2.5}
\end{equation}
Thus
\begin{equation}
\begin{split}
||F||_4^4 =& \frac{6}{\pi} + \frac{2\pi^2}{L(1,{\rm sym}^2f)^2}\sum_{m,n,r\ge 1}\frac{A_f(n,r)V_{k,1}(m)V_{k,2}(nr^2)}{(mnr^2)^{1/2}}\\
&\times \sum_{c\ge 1}\frac{S(n,m;c)}{c}J_{2k-1}\left(\frac{4\pi \sqrt{nm}}{c}\right) + O(k^{-1/2}).\label{2eq2.6}
\end{split}
\end{equation}
When (\ref{2eq2.6}) is inserted into (\ref{eq1.6}), we see that the constant $\frac{6}{\pi}$ gives the desired main term of Theorem \ref{thm1.3}.
\section{The error term}
By Lemma \ref{2thm1.5}, the bounds (cf.\cite[(0.8)--(0.10)]{110})
\begin{equation}
k^{-\varepsilon}\ll L(1,{\rm sym}^2f)\ll k^{\varepsilon}\label{2eq1.45}
\end{equation}
and the trivial bound $||F||_4^4\ll k^{1/2+\varepsilon}$, we see that Theorem \ref{thm1.3} would follow from showing that there exists some $\delta>0$ such that for any $1\le d_1,d_2,d_3\le K^{\varepsilon}$, we have
\begin{equation}
\frac{2}{HW}\sum_{k\equiv 0 \,mod\,2}w\left(\frac{k-K}{H}\right)\frac{12}{k}\sum_{f\in B_k}L(1,{\rm sym}^2f)a_f(d_1^2d_2^2)(||F||_4^4-\frac{6}{\pi}) = O(K^{-\delta}).\label{2eq1.46}
\end{equation}
It remains to prove that for some $\delta>0$ and any $1\le d_1,d_2,d_3\le K^{\varepsilon}$, we have
\begin{equation}
\begin{split}
\frac{1}{KH}&\sum_{k \equiv 0\,mod\,2}w\left(\frac{k-K}{H}\right)\sum_{f\in B_k}\frac{1}{L(1,{\rm sym}^2f)}\\
&\times\sum_{m,n,r,c\ge 1}\frac{A_f(n,r)a_f(d_1^2d_2^2)V_{k,1}(m)V_{k,2}(nr^2)}{(mnr^2)^{1/2}}\frac{S(n,m;c)}{c}\\
&\times J_{2k-1}\left(\frac{4\pi\sqrt{nm}}{c}\right) \ll K^{-\delta}.\label{2eq3.1}
\end{split}
\end{equation}
Expressing the GL$(3)$ coefficients in terms of GL$(2)$ coefficients by using (\ref{10eq29}) and (\ref{10eq210}), we get
\begin{equation}
\begin{split}
&\sum_{n,r\ge 1}  \frac{A_f(n,r)V_{k,2}(nr^2)S(n,m;c)J_{2k-1}(4\pi\sqrt{nm}/c)}{n^{1/2}r}\\
&= \sum_{n,r,\alpha \ge 1} \frac{\mu(\alpha)A_f(n,1)A_f(r,1)S(n\alpha,m;c)J_{2k-1}(4\pi\sqrt{n\alpha m}/c)V_{k,2}(nr^2\alpha^3)}{n^{1/2}r\alpha^{3/2}}\\
&= \sum_{n,r,\alpha,\beta,\gamma\ge 1}\frac{\mu(\alpha)a_f(n^2)a_f(r^2)S(n\alpha\beta^2,m;c)J_{2k-1}(4\pi\sqrt{n\alpha\beta^2m}/c)V_{k,2}(nr^2\alpha^3\beta^2\gamma^4)}{n^{1/2}r\alpha^{3/2}\beta\gamma^2}.\label{2eq3.2}
\end{split}
\end{equation}
Finally we can combine $a_f(r^2)$ and $a_f(d_1^2d_2^2)$ using the Hecke relations (\ref{2eq1.11}) to see that to establish (\ref{2eq3.1}), it suffices to prove that
\begin{proposition}\label{2thm3.1}
For some $\delta>0$ and any integers $\alpha$,$\beta$,$\gamma$,$r_1$,$r_2\ge 1$ such that $K^{-\varepsilon} < r_1/r_2<K^{\varepsilon}$, we have
\begin{equation}
\begin{split}
\frac{1}{KH}&\sum_{k \equiv 0\,mod\,2}w\left(\frac{k-K}{H}\right)\sum_{f\in B_k}\frac{1}{L(1,{\rm sym}^2f)}\sum_{n,m,c\ge 1}\frac{a_f(n^2)a_f(r_1^2)}{(mn)^{1/2}}\\
&\times \frac{S(n\alpha\beta^2,m;c)}{c} J_{2k-1}\left(\frac{4\pi\sqrt{nm\alpha\beta^2}}{c}\right)V_{k,1}(m)V_{k,2}(nr_2^2\alpha^3\beta^2\gamma^2) \ll K^{-\delta}.\label{2eq3.3}
\end{split}
\end{equation}
\end{proposition}
\par Since by Proposition \ref{2thm3.1}, we have
\begin{equation}
\begin{split}
\frac{1}{KH}&\sum_{k \equiv 0\,mod\,2}w\left(\frac{k-K}{H}\right)\sum_{f\in B_k}\frac{1}{L(1,{\rm sym}^2f)}\\
&\times\sum_{m,n,r,c\ge 1}\frac{A_f(n,r)a_f(d_1^2d_2^2)V_{k,1}(m)V_{k,2}(nr^2)}{(mnr^2)^{1/2}}\frac{S(n,m;c)}{c} J_{2k-1}\left(\frac{4\pi\sqrt{nm}}{c}\right)\\
=&\frac{1}{KH}\sum_{k \equiv 0\,mod\,2}w\left(\frac{k-K}{H}\right)\sum_{f\in B_k}\frac{1}{L(1,{\rm sym}^2f)}
\sum_{m\ge 1}\sum_{c\ge 1}\frac{a_f(d_1^2d_2^2)V_{k,1}(m)}{m^{1/2}c}\\
& \sum_{n,r,\alpha,\beta,\gamma\ge 1}\frac{\mu(\alpha)a_f(n^2)a_f(r^2)S(n\alpha\beta^2,m;c)J_{2k-1}(4\pi\sqrt{n\alpha\beta^2m}/c)V_{k,2}(nr^2\alpha^3\beta^2\gamma^4)}{n^{1/2}r\alpha^{3/2}\beta\gamma^2}\\
=&\sum_{\alpha,\beta,\gamma\ge 1}\frac{\mu(\alpha)}{\alpha^{3/2}\beta\gamma^2}\sum_{r\ge 1}\frac{1}{r}\sum_{d|(d_1^2d_2^2,r^2)}\\
&\frac{1}{KH}\sum_{k \equiv 0\,mod\,2}w\left(\frac{k-K}{H}\right)\sum_{f\in B_k}\frac{1}{L(1,{\rm sym}^2f)}\sum_{n,m,c\ge 1}\frac{a_f(n^2)}{(mn)^{1/2}}a_f\left(\frac{d_1^2d_2^2r^2}{d^2}\right)\\
&\times \frac{S(n\alpha\beta^2,m;c)}{c} J_{2k-1}\left(\frac{4\pi\sqrt{nm\alpha\beta^2}}{c}\right)V_{k,1}(m)V_{k,2}(nr_2^2\alpha^3\beta^2\gamma^4)\\
\ll &\sum_{\alpha,\beta,\gamma\le K^{1+\varepsilon}}\frac{1}{\alpha^{3/2}\beta\gamma^2}\sum_{r\le K^{2+\varepsilon}}\frac{1}{r}\sum_{d|(d_1^2d_2^2,r^2)}K^{-\delta}\\
\ll & K^{-\delta}.\label{3eqb}
\end{split}
\end{equation}
\par Now we will prove Proposition \ref{2thm3.1}.
Denote the left hand side of (\ref{2eq3.3}) by ${\rm ET}$.
By Petersson's trace formula (Lemma \ref{2thm1.3}), we have
\begin{equation}
{\rm ET} = \mathcal{E}_1 + \mathcal{E}_2,\label{2eq3.4}
\end{equation}
where
\begin{equation}
\begin{split}
\mathcal{E}_1 = &\frac{1}{2\pi^2H}\sum_{k\equiv 0\,mod\,2}w\left(\frac{k-K}{H}\right)\frac{k-1}{K}\sum_{m,c\ge 1}\frac{S(r_1\alpha\beta^2,m;c)}{(mr_1)^{1/2}c}\\
&\times J_{2k-1}\left(\frac{4\pi\sqrt{r_1m\alpha\beta^2}}{c}\right)V_{k,1}(m)V_{k,2}(r_1r_2^2\alpha^3\beta^2\gamma^2).\label{2eq3.5}
\end{split}
\end{equation}
and
\begin{equation}
\begin{split}
\mathcal{E}_2 = & \frac{1}{\pi H}\sum_{k\equiv 0\,mod\,2}i^kw\left(\frac{k-K}{H}\right)\frac{k-1}{K}\sum_{n,m,c_1,c_2\ge 1}\frac{1}{(nm)^{1/2}}\frac{S(n^2,r_1^2;c_2)}{c_2}\\
&\times \frac{S(n\alpha\beta^2,m;c_1)}{c_1}J_{k-1}\left(\frac{4\pi nr_1}{c_2}\right)J_{2k-1}\left(\frac{4\pi \sqrt{nm\alpha\beta^2}}{c_1}\right)\\
&\times V_{k,1}(m)V_{k,2}(nr_2^2\alpha^3\beta^2\gamma^2).\label{2eq3.6}
\end{split}
\end{equation}
We split $\mathcal{E}_2$ further into three pieces as in \cite{6}
\begin{equation}
\mathcal{E}_2 = E_1 + E_2 + E_3,\label{2eq3.7}
\end{equation}
where $E_1$ consists of those terms of $\mathcal{E}_2$ with
\begin{equation}
\frac{nr_1}{c_2} >K^{2-\varepsilon},\label{2eq3.8}
\end{equation}
$E_2$ consists of those terms with
\begin{equation}
K^{4/3-\varepsilon} < \frac{nr_1}{c_2} < K^{2-\varepsilon}\quad and  \quad \frac{nr_1}{c_2K^{\varepsilon}}<\frac{\sqrt{nm\alpha\beta^2}}{c_1}<\frac{nr_1K^{\varepsilon}}{c_2},\label{2eq3.9}
\end{equation}
and $E_3$ is the rest of $\mathcal{E}_2$.
\par We first deal with $\mathcal{E}_1$.
\begin{lemma}\label{2thm3.2}
For part $\mathcal{E}_1$, we have
\begin{equation}
\mathcal{E}_1\ll K^{-\delta}.\label{2eq3.10}
\end{equation}
\end{lemma}
\begin{proof}
By using Weil's bound (\ref{2eq1.16}) and (\ref{2eq1.33}) and spliting the sum of $c_1$ into two pieces, we have
\begin{equation}
\begin{split}
\mathcal{E}_1\ll  \frac{K^{\varepsilon}}{H}&\sum_{c\ge 1, m\le K^{1+\varepsilon}}\frac{|S(r_1\alpha \beta^2,m;c)|}{m^{\frac{1}{2}}c}\left|\sum_{k\equiv 0\,mod \,2}w\left(\frac{k-K}{H}\right)J_{2k-1}\left(\frac{4\pi\sqrt{r_1m\alpha\beta^2}}{c}\right)\right|\\
= \frac{K^{\varepsilon}}{H}&\left(\sum_{c\le K^{\frac{1}{2}+\varepsilon}, m\le K^{1+\varepsilon}} + \sum_{c\ge K^{\frac{1}{2}+\varepsilon}, m\le K^{1+\varepsilon}} \right)\\
&  \frac{|S(r_1\alpha \beta^2,m;c)|}{m^{\frac{1}{2}}c}\left|\sum_{k\equiv 0\,mod \,2}w\left(\frac{k-K}{H}\right)J_{2k-1}\left(\frac{4\pi\sqrt{r_1m\alpha\beta^2}}{c}\right)\right|.\label{eq3.12}
\end{split}
\end{equation}
By Lemma \ref{2thm1.7}, the first part of (\ref{eq3.12}) is bounded by
\begin{equation}
\begin{split}
\frac{K^{\varepsilon}}{H}&\sum_{c_1\le K^{\frac{1}{2}+\varepsilon}, m\le K^{1+\varepsilon}}\frac{|S(r_1\alpha \beta^2,m;c)|}{m^{\frac{1}{2}}c}\\
&\ll \frac{K^{\varepsilon}}{H}\sum_{c_1\le K^{\frac{1}{2}+\varepsilon}, m\le K^{1+\varepsilon}}\frac{1}{(mc)^{1/2}}\sum_{d|c,d|m}d^{1/2}\\
&\ll \frac{K^{\varepsilon}}{H}\sum_{c\le K^{\frac{1}{2}+\varepsilon}}\frac{1}{c^{1/2}}\sum_{d|c}d^{1/2}\sum_{m\le K^{1+\varepsilon},d|m}\frac{1}{m^{1/2}}\ll \frac{K^{3/4+\varepsilon}}{H}\ll K^{-\delta},\label{eq3.13}
\end{split}
\end{equation}
the second part of (\ref{eq3.12}) is negligible since by partial integration we have (\ref{eq1.36}).
\end{proof}

\begin{remark}\label{3remark4}
Lemma \ref{2thm3.2} is valid only for $H = K^{\frac{3}{4}+c}$, $0<c\le 1/4$ if we apply Weil's bound directly.
However we can show Lemma \ref{2thm3.2} is valid for $H = K^{\frac{1}{2}+c}$, $0<c\le 1/2$ by applying Possion summation formula for the sum of $m$ as follows.
\par Similar to the proof of Lemma \ref{2thm3.2}, we have
\begin{equation}
\begin{split}
\mathcal{E}_1 = &\frac{1}{2\pi^2H}\sum_{k\equiv 0\,mod\,2}w\left(\frac{k-K}{H}\right)\frac{k-1}{K}\sum_{m,c\ge 1}\frac{S(r_1\alpha\beta^2,m;c)}{(mr_1)^{1/2}c}\\
&\times J_{2k-1}\left(\frac{4\pi\sqrt{r_1m\alpha\beta^2}}{c}\right)V_{k,1}(m)V_{k,2}(r_1r_2^2\alpha^3\beta^2\gamma^2).\\
=& \frac{1}{2\pi^2H} \sum_{m\le K^{1+\varepsilon},c \le K^{1/2+ \varepsilon}} \frac{S(r_1\alpha\beta^2,m;c)}{(mr_1)^{1/2}c} \\
& \times \sum_{k\equiv 0\,mod\,2}w\left(\frac{k-K}{H}\right)J_{2k-1}\left(\frac{4\pi\sqrt{r_1m\alpha\beta^2}}{c}\right) + O(K^{-\delta}). \label{3eqa}
\end{split}
\end{equation}
By Lemma \ref{2thm1.7}, we have
\begin{equation}
\begin{split}
\mathcal{E}_1 =&  \frac{1}{8\pi^2H} \sum_{m\le K^{1+\varepsilon},c \le K^{1/2+ \varepsilon}} \frac{S(r_1\alpha\beta^2,m;c)}{(mr_1)^{1/2}c} \\
& \times w\left(\frac{4\pi\sqrt{r_1m\alpha\beta^2}-(2K-1)c}{2Hc}\right) + O(K^{-\delta}). \label{3eqa2}
\end{split}
\end{equation}
Opening the Kloosterman sum, we get
\begin{equation}
\begin{split}
\mathcal{E}_1 =&  \frac{1}{8\pi^2Hr_1^{1/2}} \sum_{c \le K^{1/2+ \varepsilon}} \frac{1}{c} {\sum_{x \, mod\, c}}^{*}e\left(\frac{r_1\alpha\beta^2\bar{x}}{c}\right)\sum_{m\le K^{1+\varepsilon}}\frac{1}{m^{1/2}}\\
& \times w\left(\frac{4\pi\sqrt{r_1m\alpha\beta^2}-(2K-1)c}{2Hc}\right)e\left(\frac{mx}{c}\right)+ O(K^{-\delta}). \label{3eqa3}
\end{split}
\end{equation}
Let $\phi$ be a smooth function compactly support on $[1,K]$, applying Possion summation formula for the sum of $m$, we have
\begin{equation}
\begin{split}
\mathcal{E}_1 =&  \frac{1}{8\pi^2Hr_1^{1/2}} \sum_{c \le K^{1/2+ \varepsilon}} \frac{1}{c} {\sum_{x \, mod\, c}}^{*}e\left(\frac{r_1\alpha\beta^2\bar{x}}{c}\right)\sum_{n \in \mathbb{Z}}\int_{-\infty}^{\infty}\frac{\phi(y)}{y^{1/2}}\\
& \times w\left(\frac{4\pi\sqrt{r_1y\alpha\beta^2}-(2K-1)c}{2Hc}\right)e\left((\frac{x}{c}+n)y\right)dy+ O(K^{-\delta}). \label{3eqa4}
\end{split}
\end{equation}
By partial integration, the integral of $m$ is $\ll \frac{1}{n^2}$, then $\mathcal{E}_1 \ll \frac{K^{1/2}}{H}+K^{-\delta}$.
So Lemma \ref{2thm3.2} is valid for $H = K^{\frac{1}{2}+c}$, $0<c\le 1/2$.
\end{remark}

\par As for $\mathcal{E}_2$ we first show the following lemma.
\begin{lemma}\label{2thm3.3}
For part $E_1$, we have
\begin{equation}
E_1\ll K^{-\delta}.\label{2eq3.12}
\end{equation}
\end{lemma}
\begin{proof}
By the definition of $E_1$, we have $nr_1<K^{2+\varepsilon}$, $c_2 < K^{\varepsilon}$.
By using Weil's bound (\ref{2eq1.16}) and (\ref{2eq1.33}) and spliting the sum of $c_1$ for two pieces, we have
\begin{equation}
\begin{split}
E_1 \ll  \frac{K^{3/2}}{H} &\sum_{c_1\ge 1}\frac{1}{c_1^{\frac{1}{2}}}\left|\sum_{k\equiv 0\,mod\,2}i^kw\left(\frac{k-K}{H}\right)J_{k-1}\left(\frac{4\pi nr_1}{c_2}\right)J_{2k-1}\left(\frac{4\pi \sqrt{nm\alpha\beta^2}}{c_1}\right)\right|\\
= \frac{K^{3/2}}{H}& \left(\sum_{c_1\le 4\pi K^{1/2+\varepsilon}} + \sum_{c_1 > 4\pi K^{1/2+\varepsilon}} \right)\\
 & \frac{1}{c_1^{\frac{1}{2}}}\left|\sum_{k\equiv 0\,mod\,2}i^kw\left(\frac{k-K}{H}\right)J_{k-1}\left(\frac{4\pi nr_1}{c_2}\right)J_{2k-1}\left(\frac{4\pi \sqrt{nm\alpha\beta^2}}{c_1}\right)\right|.\label{eq3.16}
\end{split}
\end{equation}
By Lemma \ref{2thm1.10}, the first part of (\ref{eq3.16}) is bounded by $K^{-\delta}$.
By (\ref{2eq1.48}) and (\ref{eq1.33}), the second part of (\ref{eq3.16}) is bounded by $(\frac{e}{4})^{2K}K^{-\frac{3}{4}}$.
\end{proof}
\begin{remark}\label{3remark3}
Lemma \ref{2thm3.3} is valid only for $H = K^{\frac{3}{4}+c}$, $0<c\le 1/4$, we can make Lemma \ref{2thm3.3} be valid for $H = K^{\frac{1}{2}+c}$, $0<c\le 1/2$ by using the same method as the Remark \ref{3remark4}.
\end{remark}
\par For the second piece we have the following lemma.
\begin{lemma}\label{2thm3.4}
For part $E_2$, we have
\begin{equation}
E_2\ll K^{-\delta}.\label{2eq3.15}
\end{equation}
\end{lemma}
\begin{proof}
By definition of $E_2$ and (\ref{2eq1.32}), we have $n< \frac{K^{2+\varepsilon}}{r_1^2}$, $m<K^{1+\varepsilon}$, $c_2>\frac{\sqrt{n}r_1c_1}{K^{\frac{1}{2}-\varepsilon}\sqrt{\alpha}\beta}$, $c_1 < K^{\frac{1}{6}+\varepsilon}$.
\par Let
\begin{equation}
\frac{4\pi \sqrt{nm_0\alpha\beta^2}}{c_1} = \frac{16\pi nr_1}{c_2}. \label{eq3.18}
\end{equation}
Denote $S_1$ the contribution to $E_2$ of the term with
\begin{equation}
|m-m_0| < m_0 \cdot \min \left\{\frac{1}{2},\frac{nr_1}{c_2K^{199/100}}\right\},\label{eq3.19}
\end{equation}
so that
\begin{equation}
\frac{1}{\sqrt{m}}\ll \frac{1}{\sqrt{m_0}} = \frac{c_2\sqrt{\alpha}\beta}{4c_1r_1\sqrt{n}}.\label{eq3.20}
\end{equation}
The rest of $E_2$ denote by $S_2$.
\par Then by (\ref{2eq1.16}), (\ref{eq3.20}), (\ref{2eq1.48}) and relation (\ref{eq3.19}), we have
\begin{equation}
\begin{split}
S_1\ll & \frac{1}{H}\sum_{n< \frac{K^{2+\varepsilon}}{r_1^2}}\sum_{m}\sum_{c_2>\frac{\sqrt{n}r_1c_1}{K^{\frac{1}{2}-\varepsilon}\sqrt{\alpha}\beta}}\sum_{c_1 < K^{\frac{1}{6}+\varepsilon}}\frac{1}{(nmc_1c_2)^{1/2}}\\
&\left|\sum_{k\equiv 0\,mod\,2}i^kw\left(\frac{k-K}{H}\right)J_{k-1}\left(\frac{4\pi nr_1}{c_2}\right)J_{2k-1}\left(\frac{4\pi\sqrt{nm\alpha\beta^2}}{c_1}\right)\right|\\
&\ll \sum_{n< \frac{K^{2+\varepsilon}}{r_1^2}}\sum_{c_2>\frac{\sqrt{n}r_1c_1}{K^{\frac{1}{2}-\varepsilon}\sqrt{\alpha}\beta}}\sum_{c_1 < K^{\frac{1}{6}+\varepsilon}} \frac{n^2r_1^3c_1^2}{c_2^3K^{199/100}}\frac{c_2}{nr_1}\frac{1}{(nc_1c_2)^{1/2}}\frac{c_2}{c_1\sqrt{n}r_1}\\
&\ll \sum_{n< \frac{K^{2+\varepsilon}}{r_1^2}}\sum_{c_1 < K^{\frac{1}{6}+\varepsilon}} \frac{K^{1/4}}{n^{1/4}}\frac{r_1^{1/2}}{K^{199/100}}\ll K^{-\frac{6}{25}}.\label{eq3.21}
\end{split}
\end{equation}
By \cite[Lemma 3.4]{6}, we see that $S_2$ is bounded by any negative power of $K$.
\end{proof}
\begin{remark}\label{3remark2}
Lemma \ref{2thm3.4} is valid for all $H$ with $0<H\le K$.
\end{remark}
\par It remains to show that $E_3$ is less than a negative power of $K$.
\begin{lemma}\label{2thm3.5}
For part $E_3$, we have
\begin{equation}
E_3\ll K^{-\delta}.\label{2eq3.26}
\end{equation}
\end{lemma}
\begin{proof}
We use another smooth, non-negative function $w '$ supported on $(1,\frac{K+H}{K})$ here, and we have
\begin{equation}
\begin{split}
E_3 = & \frac{1}{\pi H}\sum_{k\equiv 0\,mod\,2}i^kw'\left(\frac{k}{K}\right)\frac{k-1}{K}\sum_{n,m,c_1,c_2}\frac{1}{(nm)^{1/2}}\frac{S(n^2,r_1^2;c_2)}{c_2}\\
&\times \frac{S(n\alpha\beta^2,m;c_1)}{c_1}J_{k-1}\left(\frac{4\pi nr_1}{c_2}\right)J_{2k-1}\left(\frac{4\pi \sqrt{nm\alpha\beta^2}}{c_1}\right)\\
&\times V_{k,1}(m)V_{k,2}(nr_2^2\alpha^3\beta^2\gamma^2),\label{6eq3.6}
\end{split}
\end{equation}
where the terms in the sums of $n,m,c_1,c_2$ need to satisfy
\begin{equation}
\frac{nr_1}{c_2} < K^{4/3-\varepsilon}\quad or\quad \frac{nr_1}{c_2} > K^{2-\varepsilon}
\end{equation}
and
\begin{equation}
\frac{\sqrt{nm\alpha\beta^2}}{c_1} < \frac{nr_1K^{\varepsilon}}{c_2} \quad or\quad \frac{\sqrt{nm\alpha\beta^2}}{c_1} > \frac{nr_1}{c_2K^{\varepsilon}}.
\end{equation}
By \cite[Lemma 3.5]{6}, we have
\begin{equation}
E_3\ll K^{-1/4+\varepsilon}\cdot\frac{K}{H}\ll K^{-\delta}.\label{eq3.23}
\end{equation}
\end{proof}
\begin{remark}\label{3remark1}
Lemma \ref{2thm3.5} are valid for $H = K^{\frac{3}{4}+c}$, $0<c\le 1/4$.
Because we use a smooth function different from \cite{6}, after applying Poisson summation formula, the properties of the function on the index are very poor, so that it is difficult to get power saving by using stationary phase method or other methods.
That is how $H$ is restricted in Theorem \ref{thm1.3}.
\end{remark}

\section*{\bf Acknowledgements}
I am grateful to my supervisor, Prof. Yongxiao Lin, for providing the topic, careful guidance and valuable suggestions. I would like to thank Prof. Bingrong Huang, for providing the idea of applying Poisson summation formula for the sum of Kloosterman sums (See Remark \ref{3remark4}). Thanks also to Prof. Hongbo Yin for providing much encouragement and guidance.

\hspace*{\fill}
\begin{flushleft}
Jinghai Liu \\
Data Science Institute, Shandong University, Jinan 250100, Shandong, China\\
Email address: \textcolor{blue}{ljhailhy@outlook.com}
\end{flushleft}

\newpage


\begin{thebibliography}{90}

\bibitem{101} Berry, M. V., ``Regular and irregular semiclassical wavefunctions.'' J. Phys. A   10 , no. 12, 2083–2091, (1977).\\
\textcolor{blue}{http://stacks.iop.org/0305-4470/10/2083}


\bibitem{1} Blomer, V., Khan, R., Young, M., ``Distribution of mass of holomorphic cusp forms.'' Duke Math. J.162, no.14, 2609–2644, (2013).\\
\textcolor{blue}{https://doi.org/10.1215/00127094-2380967}


\bibitem{29} Buttcane, J.,  Khan, R., ``On the fourth moment of Hecke-Maass forms and the random wave conjecture.'' Compos. Math.153, no.7, 1479–1511, (2017).\\
\textcolor{blue}{https://doi.org/10.1112/S0010437X17007199}


\bibitem{2} Goldfeld, D., ``Automorphic Forms and L-Functions for the Group GL(n,R).'' Cambridge Studies in Advanced Mathematics, vol. 99. Cambridge University Press, Cambridge (2006). With an appendix by Kevin A. Broughan. \\
\textcolor{blue}{https://doi.org/10.1017/CBO9780511542923}

\bibitem{110} Hoffstein, J., Lockhart, P., ``Coefficients of Maass forms and the Siegel zero.'' With an appendix by Dorian Goldfeld, Hoffstein and Daniel Lieman. Ann. of Math. (2)   140 , no. 1, 161–181, (1994).\\
\textcolor{blue}{https://doi.org/10.2307/2118543}

\bibitem{3} Holowinsky, R., Soundararajan, K., ``Mass equidistribution for Hecke eigenforms.'' Ann. Math. 172(2), 1517-1528, (2010).\\
\textcolor{blue}{http://doi.org/10.4007/annals.2010.172.1517}

\bibitem{4} Iwaniec, H., ``Topics in Classical Automorphic Forms.'' Graduate Studies in Mathematics, vol. 17. Am. Math. Soc., Providence, (1997). \\
\textcolor{blue}{https://doi.org/10.1090/gsm/017}

\bibitem{5} Iwaniec, H., Kowalski, E. ``Analytic Number Theory.'' American Mathematical Society Colloquium Publications, vol. 53. Am. Math. Soc., Providence, (2004).\\
\textcolor{blue}{https://doi.org/10.1090/coll/053}

\bibitem{6} Khan, R., ``On the fourth moment of holomorphic Hecke cusp forms." Ramanujan J. 34, no. 1 , 83--107, (2014) .\\
\textcolor{blue}{https://doi.org/10.1007/s11139-013-9505-z}

\bibitem{100} Lindenstrauss, E., ``Invariant measures and arithmetic quantum unique ergodicity." Ann. of Math. (2)   163, no. 1, 165–219, (2006).\\
\textcolor{blue}{https://doi.org/10.4007/annals.2006.163.165}



\bibitem{7} Rankin, R., ``The vanishing of Poincar${\rm \acute{e}}$ series.'' Proc. Edinb. Math. Soc. 23(2), 151–161, (1980). \\
\textcolor{blue}{https://doi.org/10.1017/S0013091500003035}

\bibitem{31} Rudnick, Z., Sarnak, P., ``The Behaviour of Eigenstates of Arithmetic Hyperbolic Manifolds.'' Comm. Math. Phys. 161, no.1, 195–213, (1994). \\
\textcolor{blue}{http://projecteuclid.org/euclid.cmp/1104269797}


\bibitem{30} Soundararajan, K., ``Quantum unique ergodicity for ${\rm SL}_2(\mathbb{Z})\backslash \mathbb{H}$.'' Ann. of Math. (2)172, no.2, 1529–1538, (2010).\\
\textcolor{blue}{https://doi.org/10.4007/annals.2010.172.1529}


\bibitem{8} Watson, T., ``Rankin triple products and quantum chaos.'' Thesis (Ph.D.)–Princeton University
ProQuest LLC, Ann Arbor, MI, (2002). \\
\textcolor{blue}{http://arxiv.org/abs/0810.0425}

\bibitem{21} Zelditch, S., ``Mean Lindel${\rm \ddot{o}}$f hypothesis and equidistribution of cusp forms and Eisenstein series.'' J. Funct. Anal. 97, 1–49, (1991).\\
\textcolor{blue}{https://doi.org/10.1016/0022-1236(91)90014-V}

\bibitem{25} Zenz, P., ``Sharp bound for the fourth moment of holomorphic Hecke cusp forms.'' Int. Math. Res. Not. IMRN, no.16, 13562–13600, (2023).\\
\textcolor{blue}{https://doi.org/10.1093/imrn/rnac199}


\end{thebibliography}
\end{document}